\UseRawInputEncoding
\documentclass[a4paper,11pt]{amsart}
\usepackage{graphicx}
\usepackage{mathptmx}
\usepackage{amsmath}
\usepackage{amssymb}
\usepackage{enumitem}
\usepackage{xcolor}

\newtheorem{theorem}{Theorem}
\newtheorem{lemma}[theorem]{Lemma}
\newtheorem{corollary}[theorem]{Corollary}

\begin{document}

\title{A new type of  degenerate  poly-Euler  polynomials}

\author{Yuankui Ma$^{1,*}$}
\address{1. School of Science, Xi'an Technological University, Xi'an, 710021, Shaanxi, P. R. China}
\email{mayuankui@xatu.edu.cn}

\author{Taekyun  Kim$^{1, 2}$}
\address{1. School of Science, Xi'an Technological University, Xi'an, 710021, Shaanxi, P. R. China\\
2. Department of Mathematics, Kwangwoon University, Seoul 139-701, Republic of Korea}
\email{tkkim@kw.ac.kr}

\author{Hongze  Li$^{3}$}
\address{3. School of Mathematical Sciences, Shanghai Jiao Tong University, Shanghai 200240, People¡¯s Republic of China}
\email{lihz@sjtu.edu.cn}


\thanks{* is corresponding author}

\begin{abstract}
Many mathematicians have been studying various degenerate versions of special polynomials and numbers in some arithmetic and combinatorial aspects.
Our main focus here is a new type of degenerate poly-Euler polynomials and numbers. This focus stems from their nascent importance for applications in combinatorics, number theory and in other aspects of applied mathematics.
we  construct a new type of  degenerate poly-Euler polynomials by using  the degenerate polylogarithm functions. We also show several combinatorial identities related to  this polynomials and numbers.

\end{abstract}

\maketitle

\markboth{\centerline{\scriptsize A new type of degenerate  poly-Euler polynomials}}
{\centerline{\scriptsize   Yuankui Ma, Taekyun  Kim and Hongze  Li}}

\section{Introduction}
%%%%%%%%%%%%%%%%%%%%%%%%%%%%%%%%%%%%%

%%%%%%%%%%%%%%%%%%%%%%%%%%%%%%%

Recently, many mathematicians have  been  studying various degenerate versions of special polynomials and numbers in some arithmetic and combinatorial aspects [6, 9 - 20].
%(or Studying degenerate versions of various special polynomials become an active area of research and yielded many interesting arithmetic and combinatorial results.)
These degenerate versions began when Carlitz introduced  the degenerate  Bernoulli polynomials and the degenerate Euler polynomials [1]. These polynomials appear in combinatorial mathematics and play a very important role in the
theory and application of mathematics, thus many number theory and combination experts have
studied their properties, and obtained a series of interesting results. Kim et al introduced degenerate gamma random variables as well as
 new  Jindalrae and Gaenari numbers and polynomials, and developed   above mentioned polynomials and numbers related to  Jindalrae and Gaenari numbers and polynomials; discrete harmonic numbers [18 - 21].
Motivated by their importance and potential for applications in number theory, combinatorics  and other fields of applied mathematics, in particular, we are interested in  degenerate poly-Euler polynomials and numbers.
The goal of this paper is to demonstrate many explicit computational formulas and relations,  involving a new type of the degenerate poly-Euler polynomials and numbers by using Kim-Kim's the degenerate polylogarithm functions.

Now, we give some definitions and properties needed in this paper.
As is known,
the ordinary Euler polynomials and the ordinary Bernoulli polynomials are usually defined by the following  generating functions with parallel structures [1 - 2],
\begin{equation}\label{1}
\begin{split}
\frac{t}{e^t-1}e^{xt}=\sum_{n=0}^\infty B_n(x)\frac{t^n}{n!}, \ \   \ \ \frac{2}{e^t+1}e^{xt} = \sum_{n=0} ^\infty E_n(x)\frac{t^n}{n!}.
\end{split}
\end{equation}

For any nonzero $\lambda\in\mathbb{R}$ (or $\mathbb{C}$), the degenerate exponential function is defined by

\begin{equation}\label{2}
\begin{split}
	e_{\lambda}^{x}(t)=(1+\lambda t)^{\frac{x}{\lambda}},\quad e_{\lambda}(t)=(1+\lambda t)^{\frac{1}{\lambda}}=e_{\lambda}^{1}(t),\; \;  \quad
\text{(see [11 - 15, 20])}.
\end{split}
\end{equation}

By Taylor expansion, we get

\begin{equation}\label{3}
\begin{split}
	e_{\lambda}^{x}(t)=\sum_{n=0}^{\infty}(x)_{n,\lambda}\frac{t^{n}}{n!},\; \;  \quad
\text{(see [11 - 15])},
\end{split}
\end{equation}
where $\displaystyle (x)_{0,\lambda}=1,\ (x)_{n,\lambda}=x(x-\lambda)(x-2\lambda)\cdots(x-(n-1)\lambda),\ (n\ge 1)$.

Note that
\begin{equation*}
\begin{split}
	\lim_{\lambda\rightarrow 0}e_{\lambda}^{x}(t)=\sum_{n=0}^{\infty}\frac{x^{n}t^{n}}{n!}=e^{xt}.
\end{split}
\end{equation*}

Carlitz [1] introduced the ordinary degenerate  Bernoulli polynomials and the degenerate  Euler polynomials, respectively given by

\begin{equation}\label{4}
\begin{split}
\frac{2}{e_\lambda(t)+1}e_\lambda^{x}(t) = \sum_{n=0} ^\infty E_{n,\lambda}(x)\frac{t^n}{n!}, \  \    \ \ \frac{t}{e_\lambda(t)-1}e^x_\lambda(t)=\sum_{n=0} ^\infty B_{n,\lambda}(x)\frac{t^n}{n!}.
\end{split}
\end{equation}

In 2020, Kim-Kim [9] introduced  the degenerate polylogarithm function  defined by
\begin{equation}\label{5}
	l_{k,\lambda}(x)=\sum_{n=1}^{\infty}\frac{(-\lambda)^{n-1}(1)_{n,1/\lambda}}{(n-1)!n^{k}}x^{n}, k\in\mathbb{Z}\quad (|x|<1).
\end{equation}
We note that
\begin{equation*}
\begin{split}
	\lim_{\lambda\rightarrow 0}l_{k,\lambda}(x)=\sum_{n=1}^{\infty}\frac{x^{n}}{n^{k}}=Li_{k}(x).
\end{split}
\end{equation*}

From \eqref{5}, we have
\begin{equation}\label{6}
\begin{split}
	 l_{1,\lambda}(x)=\sum_{n=1}^{\infty}\frac{(-\lambda)^{n-1}(1)_{n,1/\lambda}}{n!}x^{n}=-\log_{\lambda}(1-x).
\end{split}
\end{equation}

Kim-Kim also studied the new type degenerate poly Bernoulli polynomials and numbers, by using the degenerate polylogarithm function as follows:

\begin{equation}\label{7}
\begin{split}
	 \frac{l_{k,\lambda}(1-e_{\lambda}(-t))}{1-e_{\lambda}(-t)}e_{\lambda}^{x}(-t)=\sum_{n=0}^{\infty}\beta_{n,\lambda}^{(k)}(x)\frac{t^{n}}{n!},\; \;  \quad
\text{(see [9])}.
\end{split}
\end{equation}
When $x=0$, $\beta_{n,\lambda}^{(k)}=\beta_{n,\lambda}^{(k)}(0)$ are called the degenerate poly-Bernoulli numbers.\\

 For $n\geq 0$, the Stirling numbers of the first kind are defined by
\begin{equation}\label{8}
\begin{split}
(x)_n=\sum_{l=0}^n S_1(n,l) x^l,\; \;  \quad
\text{(see [8 - 22])}.
\end{split}
\end{equation}
where $(x)_0=1, \; (x)_n=x(x-1)\dots (x-n+1)$, $(n\geq 1)$.\\

From \eqref{8}, it is easy to see that
\begin{equation}\label{9}
\frac{1}{k!} (\log(1+t))^k
=\sum_{n=k}^\infty S_1(n,k)\frac{t^n}{n!},\; \;  \quad
\text{(see [3, 22])}.
\end{equation}

In the inverse expression to \eqref{8}, for $n\geq 0$, the Stirling numbers of the second kind are defined by
\begin{equation}\label{10}
\begin{split}
x^n=\sum_{l=0}^n S_2(n,l) (x)_l,\; \;  \quad
\text{(see [3, 22])} .
\end{split}
\end{equation}
From \eqref{10}, it is easy to see that
\begin{equation}\label{11}
\begin{split}
\frac{1}{k!} (e^t -1)^k
=\sum_{n=k}^\infty S_2(n,k)\frac{t^n}{n!},\; \;  \quad
\text{(see [2])}.
\end{split}
\end{equation}

 Kim et al. [13] introduced the degenerate Stirling numbers of the second kind as follows:
\begin{equation}\label{12}
\begin{split}
	 (x)_{n,\lambda}=\sum_{l=0}^{n}S_{2,\lambda}(n,l)(x)_{l}\quad (n\ge 0).
\end{split}
\end{equation}
As an inversion formula of \eqref{12}, the degenerate Stirling numbers of the first kind are defined by
\begin{equation}\label{13}
\begin{split}
	(x)_{n}=\sum_{l=0}^{n}S_{1,\lambda}(n,l)(x)_{l,\lambda}\quad (n\ge 0),\; \;  \quad
\text{(see [9]}.
\end{split}
\end{equation}
From \eqref{12} and \eqref{13}, it is well known that
\begin{equation}\label{14}
\begin{split}
	\frac{1}{k!}\big(e_{\lambda}(t)-1\big)^{k}=\sum_{n=k}^{\infty}S_{2,\lambda}(n,k)\frac{t^{n}}{n!} \quad(k\geq0),\; \;  \quad
\text{(see [9, 13])},
\end{split}
\end{equation}
and
\begin{equation}\label{15}
\begin{split}
	\frac{1}{k!}\big(\log_{\lambda}(1+t)\big)^{k}=\sum_{n=k}^{\infty}S_{1,\lambda}(n,k)\frac{t^{n}}{n!} \quad(k\geq0),\; \;  \quad
\text{(see [9])}.
\end{split}
\end{equation}

In 1997, Kaneko [7] introduced the poly-Bernoulli numbers $B_n^{(k)}$  represented by the following  generating function
\begin{equation*}
\begin{split}
\frac{Li_k(1-e^{-t})}{e^t-1}= \sum_{n=0}^\infty B_n^{(k)}  \frac{t^n}{n!},
\end{split}
\end{equation*}
where
\begin{equation}
\begin{split}
Li_k(z)=\sum_{n=0}^\infty \frac{z^n}{n^k} \quad \text{(see [24])}.
\end{split}
\end{equation}
If $k=1$, we get $B_n^{(1)}=(-1)B_n$ for $n\geq0$ where $B_n$ are the Bernoulli numbers.

The poly-Bernoulli polynomials of index k were defined by the generating function
\begin{equation}\label{17}
\begin{split}
\frac{Li_k(1-e^{-t})}{e^t-1}e^{xt}= \sum_{n=0}^\infty \beta_n^{(k)} (x) \frac{t^n}{n!} \quad  \text{(see [25])}.
\end{split}
\end{equation}

Ohno and Sasaki [26] defined poly-Euler numbers as

\begin{equation*}
\begin{split}
\frac{Li_k(1-e^{-4t})}{4t cosh(t)}= \sum_{n=0}^\infty E_n^{(k)} (x) \frac{t^n}{n!}.
\end{split}
\end{equation*}

It was recently extended to
\begin{equation*}
\begin{split}
\frac{2Li_k(1-e^{-t})}{1+e^t}e^{xt}= \sum_{n=0}^\infty E_n^{(k)} (x) \frac{t^n}{n!},
\end{split}
\end{equation*}
in polynomial form by Yoshinori [27].

 Lee et al. [23] introduced  the type 2 degenerate poly-Euler polynomials  constructed from
the modified polyexponential function
\begin{equation}\label{18}
\begin{split}
\frac{Ei_{k}\big(log(1+2t)\big)}{t(e_{\lambda}(t)+1)}e_{\lambda}^{x}(t)=\sum_{n=0}^{\infty}E_{n,\lambda}^{(k)}(x)\frac{t^{n}}{n!}.
\end{split}
\end{equation}

An outline of this paper is as follows.
In section 2,
 we  construct a new type of poly-Euler polynomials and numbers, by using the polylogarithm function. We also show several combinatorial identities related to  the poly-Euler polynomials and numbers. Some of them include some special polynomials and numbers such as the Stirling numbers of the second kind, ordinary Euler numbers, ordinary Bernoulli polynomials and numbers, poly Bernoulli polynomials, etc.
In  section 3, we also consider the degenerate poly-Euler polynomials, by using the degenerate polylogarithm function and investigate some identities for those polynomials.

\section{A new type of poly-Euler numbers and polynomials}

In this section, we define a new type of  poly-Euler polynomials and numbers, by using the polylogarithm functions. We also show several combinatorial identities related to the poly-Euler polynomials and numbers.

The new type of  poly-Euler polynomials is defined by
\begin{equation}\label{19}
\begin{split}
\frac{Li_{k}(1-e^{-2t})}{t(e^{t}+1)}e^{xt}=\sum_{n=0}^{\infty}E_{n}^{(k)}(x)\frac{t^{n}}{n!}.
\end{split}
\end{equation}

When $x=0$, $E_{n}^{(k)}=E_{n}^{(k)}(0)$ are called the poly-Euler numbers.

When $k=1$, as $Li_{1}(1-e^{-2t})=-\log(1-1+e^{-2t})=2t$, we see that  $ E_{n}^{(1)}(x)=E_{n}(x)$ $(n \geq 0)$ are the Euler polynomials.

\begin{lemma} For $n\geq 0,~ k\in \mathbb{Z}$, we have
\begin{equation}\label{20}
\begin{split}
Li_{k}(1-e^{-2t})=\sum_{n=1}^{\infty}\bigg(\sum_{m=1}^{n}\frac{2^n(-1)^{n+m}m!}{m^{k}}S_{2}(n,m)\bigg)\frac{t^{n}}{n!}.
\end{split}
\end{equation}
\end{lemma}

\begin{proof}
\begin{align*}
Li_{k}\big(1-e^{-2t}\big)\ &=\ \sum_{m=1}^{\infty}\frac{1}{m^{k}}\big(1-e^{-2t}\big)^{m}\ =\ \sum_{m=1}^{\infty}\frac{(-1)^{m}m!}{m^{k}}\frac{1}{m!}\big(e^{-2t}-1\big)^{m} \nonumber \\
&=\ \sum_{m=1}^{\infty}\frac{(-1)^{m}m!}{m^{k}}\sum_{n=m}^{\infty}S_{2}(n,m)(-2)^{n}\frac{t^{n}}{n!}  \\
&=\ \sum_{n=1}^{\infty}\bigg(\sum_{m=1}^{n}\frac{2^n(-1)^{n+m}m!}{m^{k}}S_{2}(n,m)\bigg)\frac{t^{n}}{n!}.
\end{align*}
\end{proof}

\begin{theorem} For $n\geq 0,~ k\in \mathbb{Z}$, we have
\begin{equation}\label{21}
\begin{split}
E_{n}^{(k)}(x)=\sum_{l=0}^{n}\sum_{m=0}^{l}\sum_{j=1}^{m+1}\binom{n}{l}\binom{l}{m}\frac{2^{m+n-l}(-1)^{m+1+j}j!}{(l-m+1)j^{k}(m+1)}S_{2}(m+1,j)B_{n-l}(\frac{x}{2}).
\end{split}
\end{equation}
\end{theorem}

\begin{proof}
From \eqref{1} and \eqref{20}, we have
\begin{equation}\label{22}
\begin{split}
\sum_{n=0}^{\infty}E_{n}^{(k)}(x)\frac{t^{n}}{n!}
&=\frac{Li_{k}(1-e^{-2t})}{t(e^{t}+1)}e^{xt}\\
&=\frac{e^{xt}}{t(e^{2t}-1)} (e^{t}-1)\sum_{m=1}^{\infty}\bigg(\sum_{j=1}^{m}\frac{2^m(-1)^{m+j}j!}{j^{k}}S_{2}(m,j)\bigg)\frac{t^{m}}{m!}\\
&=\frac{e^{xt}}{t(e^{2t}-1)} (\sum_{i=1}^{\infty}\frac{t^{i}}{i!})\sum_{m=0}^{\infty}\bigg(\sum_{j=1}^{m+1}\frac{2^{m+1}(-1)^{m+1+j}j!}{j^{k}}S_{2}(m+1,j)\bigg)\frac{t^{m+1}}{(m+1)!}\\
&=\frac{2te^{xt}}{e^{2t}-1} \bigg(\sum_{i=0}^{\infty}\frac{1}{i+1}\frac{t^{i}}{i!}\bigg)\sum_{m=0}^{\infty}\bigg(\sum_{j=1}^{m+1}\frac{2^{m}(-1)^{m+1+j}j!}{j^{k}}S_{2}(m+1,j)\bigg)\frac{t^{m}}{(m+1)!}\\
&=\bigg(\sum_{s=0}^{\infty}B_{s}(\frac{x}{2})\frac{(2t)^{s}}{s!} \bigg) \bigg(\sum_{l=0}^{\infty}\sum_{m=0}^{l}\sum_{j=1}^{m+1}\binom{l}{m}\frac{2^{m}(-1)^{m+1+j}j!}{(l-m+1)j^{k}(m+1)}S_{2}(m+1,j)\frac{t^{l}}{l!}\bigg)\\
&=\sum_{n=0}^{\infty}\bigg( \sum_{l=0}^{n}\sum_{m=0}^{l}\sum_{j=1}^{m+1}\binom{n}{l}\binom{l}{m}\frac{2^{m+n-l}(-1)^{m+1+j}j!}{(l-m+1)j^{k}(m+1)}S_{2}(m+1,j)B_{n-l}(\frac{x}{2})\bigg)\frac{t^{n}}{n!}.
\end{split}
\end{equation}
Therefore, by comparing the coefficients on both sides of \eqref{22}, we get the desired result.
\end{proof}

\begin{corollary} For $n\geq 0,~ k\in \mathbb{Z}$, and $x=0$, we have
\begin{equation*}
\begin{split}
E_{n}^{(k)}=\sum_{l=0}^{n}\sum_{m=0}^{l}\sum_{j=1}^{m+1}\binom{n}{l}\binom{l}{m}\frac{2^{m+n-l}(-1)^{m+1+j}j!}{(l-m+1)j^{k}(m+1)}S_{2}(m+1,j)B_{n-l}.
\end{split}
\end{equation*}
\end{corollary}

\begin{theorem} For $n\geq 0,~ k\in \mathbb{Z}$, we have
\begin{equation}\label{23}
\begin{split}
E_{n}^{(k)}(x)=\sum_{m=0}^{n}\sum_{l=1}^{m+1}\binom{n}{m}\frac{2^{m}(-1)^{m+1+l}l!}{l^{k}(m+1)}S_{2}(m+1,l)E_{n-m}(x).
\end{split}
\end{equation}
\end{theorem}

\begin{proof}From \eqref{1} and \eqref{20}, we have
\begin{equation}\label{24}
\begin{split}
\sum_{n=0}^{\infty}E_{n}^{(k)}(x)\frac{t^{n}}{n!}
&=\frac{Li_{k}(1-e^{-2t})}{t(e^{t}+1)}e^{xt}\\
&=\frac{e^{xt}}{t(e^{t}+1)} \sum_{m=1}^{\infty}\bigg(\sum_{l=1}^{m}\frac{2^m(-1)^{m+l}l!}{l^{k}}S_{2}(m,l)\bigg)\frac{t^{m}}{m!}\\
&=\frac{2e^{xt}}{e^{t}+1} \sum_{m=0}^{\infty}\bigg(\sum_{l=1}^{m+1}\frac{2^{m}(-1)^{m+1+l}l!}{l^{k}}S_{2}(m+1,l)\bigg)\frac{t^{m}}{(m+1)!}\\
&= \bigg(\sum_{j=0}^{\infty}E_{j}(x)\frac{t^{j}}{j!}\bigg)\sum_{m=0}^{\infty}\bigg(\sum_{j=1}^{m+1}\frac{2^{m}(-1)^{m+1+l}l!}{l^{k}(m+1)}S_{2}(m+1,l)\bigg)\frac{t^{m}}{m!}\\
&= \sum_{n=0}^{\infty}\bigg(\sum_{m=0}^{n}\sum_{l=1}^{m+1}\binom{n}{m}\frac{2^{m}(-1)^{m+1+l}l!}{l^{k}(m+1)}S_{2}(m+1,l)E_{n-m}(x)\bigg)\frac{t^{n}}{n!}.
\end{split}
\end{equation}
Therefore, by comparing the coefficients on both sides of \eqref{24}, we get the desired result.
\end{proof}

\begin{corollary} For $n\geq 0,~ k\in \mathbb{Z}$, and $x=0$, we have
\begin{equation*}
\begin{split}
E_{n}^{(k)}=\sum_{m=0}^{n}\sum_{l=1}^{m+1}\binom{n}{m}\frac{2^{m}(-1)^{m+1+l}l!}{l^{k}(m+1)}S_{2}(m+1,l)E_{n-m}.
\end{split}
\end{equation*}
\end{corollary}

\begin{theorem} For $n\geq 0,~ k\in \mathbb{Z}$, we have
\begin{equation}\label{25}
\begin{split}
E_{n}^{(k)}(x)=\sum_{l=0}^{n}\binom{n}{l}E_{n-l}^{(k)}x^{l}.
\end{split}
\end{equation}
\end{theorem}

\begin{proof}
\begin{equation}\label{26}
\begin{split}
\sum_{n=0}^{\infty}E_{n}^{(k)}(x)\frac{t^{n}}{n!}
&=\frac{Li_{k}(1-e^{-2t})}{t(e^{t}+1)}e^{xt}\\
&= \sum_{m=0}^{\infty}E_{m}^{(k)}\frac{t^{m}}{m!}\sum_{l=0}^{\infty}\frac{(xt)^{l}}{l!}\\
&= \sum_{n=0}^{\infty}\sum_{l=0}^{n}\binom{n}{l}E_{n-l}^{(k)}x^{l}\frac{t^{n}}{n!}.
\end{split}
\end{equation}
Therefore, by comparing the coefficients on both sides of \eqref{26}, we get the desired result.
\end{proof}

\begin{theorem} For $n\geq 0,~ k\in \mathbb{Z}$, we have
\begin{equation}\label{27}
\begin{split}
\frac{d}{dx}E_{n}^{(k)}(x)&=nE_{n-1}^{(k)}(x).
\end{split}
\end{equation}
\end{theorem}

\begin{proof}From \eqref{25}, we have
\begin{equation}\label{28}
\begin{split}
\frac{d}{dx}E_{n}^{(k)}(x)&=\frac{d}{dx}\bigg(\sum_{l=0}^{n}\binom{n}{l}E_{n-l}^{(k)}x^{l}\bigg)\\
&= \sum_{l=1}^{n}\binom{n}{l}E_{n-l}^{(k)}lx^{l-1}\\
&= \sum_{l=0}^{n-1}\binom{n}{l+1}E_{n-l-1}^{(k)}(l+1)x^{l}\\
&= n\sum_{l=0}^{n-1}\frac{(n-1)!}{l!(n-l-1)!}E_{n-l-1}^{(k)}x^{l}\\
&= nE_{n-1}^{(k)}(x).
\end{split}
\end{equation}
Thus, we get the desired result.
\end{proof}

\begin{theorem}
For $n\geq 0, ~~k \in \mathbb{Z}$, we have
\begin{equation*}
\begin{split}
	 E_{n}^{(k)}(x)=\sum_{l=0}^{n}\sum_{m=0}^{l}\binom{n}{l}(x)_{m}S_{2}(m,l)E_{n-l}^{(k)}.
\end{split}
\end{equation*}	
\end{theorem}

\begin{proof}
From \eqref{11} and \eqref{19}, we get
\begin{equation}\label{29}
\begin{split}
	\sum_{n=0}^{\infty}E_{n}^{(k)}(x)\frac{t^{n}}{n!}\ &=\ \frac{Li_{k}\big(1-e^{-2t}\big)}{t(e^t+1)}\big(e^t-1+1\big)^{x}\\
    &=\ \sum_{i=0}^{\infty}E_{i}^{(k)}\frac{t^{i}}{i!} \sum_{m=0}^{\infty}(x)_{m}\frac{(e^t-1)^m}{m!} \\
	&=\ \sum_{i=0}^{\infty}E_{i}^{(k)}\frac{t^{i}}{i!} \sum_{l=0}^{\infty}\bigg(\sum_{m=0}^{l}(x)_{m}S_{2}(m,l)\bigg)\frac{t^{m}}{m!} \\
	&=\ \sum_{n=0}^{\infty}\bigg(\sum_{l=0}^{n}\sum_{m=0}^{l}\binom{n}{l}(x)_{m}S_{2}(m,l)E_{n-l}^{(k)}\bigg)\frac{t^{n}}{n!}.
\end{split}
\end{equation}
Therefore, by comparing the coefficients on both sides of \eqref{29}, we have the desired result.

\end{proof}

\begin{theorem}
For $n\geq 0, ~~k \in \mathbb{Z}$, we have
\begin{equation*}
\begin{split}
	 nE_{n-1}^{(k)}(x+1) +nE_{n-1}^{(k)}(x)=2^{n}\bigg(\beta_{n}^{(k)}(x+2) -\beta_{n}^{(k)}(x)\bigg).
\end{split}
\end{equation*}	
\end{theorem}

\begin{proof}
From \eqref{17} and \eqref{19}, we get
\begin{equation}\label{30}
\begin{split}
	 &\frac{Li_{k}\big(1-e^{-2t}\big)}{t(e^t+1)}t(e^t+1)e^{xt}\\
&=\ \frac{Li_{k}\big(1-e^{-2t}\big)}{t(e^t+1)}te^{(x+1)t}+\frac{Li_{k}\big(1-e^{-2t}\big)}{t(e^t+1)}te^{xt}\\
    &=\ \sum_{n=0}^{\infty}E_{n}^{(k)}(x+1)\frac{t^{n+1}}{n!} +\sum_{n=0}^{\infty}E_{n}^{(k)}(x)\frac{t^{n+1}}{n!} \\
	&=\ \sum_{n=1}^{\infty}\bigg(nE_{n-1}^{(k)}(x+1) +nE_{n-1}^{(k)}(x)\bigg)\frac{t^{n}}{n!}.
\end{split}
\end{equation}

On the other hand,
\begin{equation}\label{31}
\begin{split}
&\frac{Li_{k}\big(1-e^{-2t}\big)}{e^{2t}-1}(e^{2t}-1)e^{xt}\\
&=\ \frac{Li_{k}\big(1-e^{-2t}\big)}{e^{2t}-1}e^{(x+2)t}-\frac{Li_{k}\big(1-e^{-2t}\big)}{e^{2t}-1}e^{xt}\\
   	&=\ \sum_{n=0}^{\infty}\bigg(\beta_{n}^{(k)}(x+2) -\beta_{n}^{(k)}(x)\bigg)\frac{(2t)^{n}}{n!}\\
  &=\ \sum_{n=0}^{\infty}2^{n}\bigg(\beta_{n}^{(k)}(x+2) -\beta_{n}^{(k)}(x)\bigg)\frac{t^{n}}{n!}.
\end{split}
\end{equation}

Therefore, by comparing the coefficients of  \eqref{30} and \eqref{31}, we get the desired result.

\end{proof}

\section{A new type of degenerate poly-Euler numbers and polynomials}

In this section, we construct a new type of  degenerate poly-Euler polynomials,  by using the degenerate polylogarithm function. We also show several combinatorial identities related to  the degenerate poly-Euler polynomials and numbers.

We define the degenerate poly-Euler polynomials, by using the degenerate polylogarithm function as follows:
\begin{equation}\label{32}
\begin{split}
\frac{l_{k,\lambda}\big(1-e_{\lambda}(-2t)\big)}{t(e_{\lambda}(t)+1)}e_{\lambda}^{x}(t)=\sum_{n=0}^{\infty}E_{n,\lambda}^{(k)}(x)\frac{t^{n}}{n!}.
\end{split}
\end{equation}

When $x=0$, $E_{n,\lambda}^{(k)}=E_{n,\lambda}^{(k)}(0)$ are called the degenerate poly-Euler numbers.\\
When $k=1$, from \eqref{12}, we see that  $ E_{n,\lambda}^{(1)}(x)=E_{n,\lambda}(x)$ $(n \geq 0)$ are the degenerate Euler polynomials
because of
\begin{equation}\label{33}
\begin{split}
	l_{1,\lambda}\big(1-e_{\lambda}(-2t)\big)=-\log_{\lambda}(1-1+e_{\lambda}(-2t))=2t.
\end{split}
\end{equation}

\begin{theorem}
For $n\geq 0, ~~k \in \mathbb{Z}$, we have
\begin{equation*}
\begin{split}
	 E_{n,\lambda}^{(k)}(x)=\sum_{l=0}^{n}\binom{n}{l}(x)_{l,\lambda}E_{n-l,\lambda}^{(k)}.
\end{split}
\end{equation*}	
\end{theorem}

\begin{proof}
From \eqref{3} and \eqref{32}, we get
\begin{equation}\label{34}
\begin{split}
	\sum_{n=0}^{\infty}E_{n,\lambda}^{(k)}(x)\frac{t^{n}}{n!}\ &=\ \frac{l_{k,\lambda}\big(1-e_{\lambda}(-2t)\big)}{t(e_{\lambda}(t)+1)}e_{\lambda}^{x}(t)\\
    &=\ \sum_{m=0}^{\infty}E_{m,\lambda}^{(k)}\frac{t^{m}}{m!} \sum_{l=0}^{\infty}(x)_{l,\lambda}\frac{t^l}{l!} \\
	&=\ \sum_{n=0}^{\infty}\bigg(\sum_{l=0}^{n}\binom{n}{l}(x)_{l,\lambda}E_{n-l,\lambda}^{(k)}\bigg)\frac{t^{n}}{n!}.
\end{split}
\end{equation}
Therefore, by comparing the coefficients on both sides of \eqref{34}, we have the desired result.

\end{proof}

\begin{theorem}
For $n\geq 0, ~~k \in \mathbb{Z}$, we have
\begin{equation*}
\begin{split}
	 E_{n,\lambda}^{(k)}(x)=\sum_{l=0}^{n}\sum_{m=0}^{l}\binom{n}{l}(x)_{m}S_{2,\lambda}(m,l)E_{n-l,\lambda}^{(k)}.
\end{split}
\end{equation*}	
\end{theorem}

\begin{proof}
From \eqref{14} and \eqref{32}, we get
\begin{equation}\label{35}
\begin{split}
	\sum_{n=0}^{\infty}E_{n,\lambda}^{(k)}(x)\frac{t^{n}}{n!}\ &=\ \frac{l_{k,\lambda}\big(1-e_{\lambda}(-2t)\big)}{t(e_{\lambda}(t)+1)}\big(e_{\lambda}(t)-1+1\big)^{x}\\
    &=\ \sum_{i=0}^{\infty}E_{i,\lambda}^{(k)}\frac{t^{i}}{i!} \sum_{m=0}^{\infty}(x)_{m}\frac{(e_{\lambda}(t)-1)^m}{m!} \\
	&=\ \sum_{i=0}^{\infty}E_{i,\lambda}^{(k)}\frac{t^{i}}{i!} \sum_{l=0}^{\infty}\bigg(\sum_{m=0}^{l}(x)_{m}S_{2,\lambda}(m,l)\bigg)\frac{t^{m}}{m!} \\
	&=\ \sum_{n=0}^{\infty}\bigg(\sum_{l=0}^{n}\sum_{m=0}^{l}\binom{n}{l}(x)_{m}S_{2,\lambda}(m,l)E_{n-l,\lambda}^{(k)}\bigg)\frac{t^{n}}{n!}.
\end{split}
\end{equation}
Therefore, by comparing the coefficients on both sides of \eqref{35}, we have the desired result.

\end{proof}

\begin{theorem}
For $n\geq 0, k \in \mathbb{Z}$, we have
\begin{equation*}
\begin{split}
E_{n-1,\lambda}^{(k)}(1)+E_{n-1,\lambda}^{(k)}=\frac{2^{n}}{n}\sum_{l=1}^{n}\frac{(1)_{l,1/\lambda}(-1)^{n-1}\lambda^{l-1}}{l^{k-1}}S_{2,\lambda}(n,l).
\end{split}
\end{equation*}
\end{theorem}

\begin{proof}
From \eqref{5}, \eqref{14} and \eqref{32},  we have
\begin{equation}\label{36}
\begin{split}
l_{k,\lambda}\big(1-e_{\lambda}(-2t)\big)
&= {t(e_{\lambda}(t)+1)}\sum_{l=0}^{\infty}E_{l,\lambda}^{(k)}\frac{t^{l}}{l!}\\
&= {t\bigg(\sum_{m=0}^{\infty}(1)_{m,\lambda}\frac{t^{m}}{m!}+1\bigg)}\sum_{l=0}^{\infty}E_{l,\lambda}^{(k)}\frac{t^{l}}{l!}\\
&= t\bigg(\sum_{n=0}^{\infty}\bigg(\sum_{l=0}^{n}\binom{n}{l}(1)_{n-l,\lambda}E_{l,\lambda}^{(k)}+E_{n,\lambda}^{(k)}\bigg)\frac{t^{n}}{n!}\bigg)\\
&= \sum_{n=0}^{\infty}\bigg(\sum_{l=0}^{n}\binom{n}{l}(1)_{n-l,\lambda}E_{l,\lambda}^{(k)}+E_{n,\lambda}^{(k)}\bigg)(n+1)\frac{t^{n+1}}{(n+1)!}\\
&= \sum_{n=1}^{\infty}n\bigg(\sum_{l=0}^{n-1}\binom{n-1}{l}(1)_{n-1-l,\lambda}E_{l,\lambda}^{(k)}+E_{n-1,\lambda}^{(k)}\bigg)\frac{t^{n}}{n!}\\
&= \sum_{n=1}^{\infty}n\bigg(E_{n-1,\lambda}^{(k)}(1)+E_{n-1,\lambda}^{(k)}\bigg)\frac{t^{n}}{n!}.
\end{split}
\end{equation}

On the other hand,
\begin{equation}\label{37}
\begin{split}
l_{k,\lambda}\big(1-e_{\lambda}(-2t)\big)
&= \sum_{l=1}^{\infty}\frac{(1)_{l,1/\lambda}(-\lambda)^{l-1}}{(l-1)!l^{k}}\big(1-e_{\lambda}(-2t)\big)^{l}\\
&= \sum_{l=1}^{\infty}\frac{(1)_{l,1/\lambda}(-\lambda)^{l-1}(-1)^{l}}{l^{k-1}}\frac{\big(e_{\lambda}(-2t)-1\big)^{l}}{l!}\\
&= \sum_{l=1}^{\infty}\frac{(1)_{l,1/\lambda}(-\lambda)^{l-1}(-1)^{l}}{l^{k-1}}\sum_{n=l}^{\infty}S_{2,\lambda}(n,l)\frac{(-2t)^n}{n!}\\
&= \sum_{n=1}^{\infty} \bigg(\sum_{l=1}^{n}\frac{(1)_{l,1/\lambda}(-1)^{2l+n-1}\lambda^{l-1}2^n}{l^{k-1}}S_{2,\lambda}(n,l)\bigg)\frac{t^n}{n!}.
\end{split}
\end{equation}
Therefore, by comparing the coefficients of  \eqref{36} and \eqref{37}, we have the desired result.
\end{proof}

\medskip

\begin{theorem}For $n\geq 0, ~~k \in \mathbb{Z}$, we have
\begin{equation*}
\begin{split}
&\sum_{i=1}^{n}\sum_{m=0}^{n-i}\binom{n}{i} (1)_{i,\lambda}2^{n}(-1)^{m+n+1} S_{2,\lambda}(n-i,m) E_{m,\lambda}^{(k)}\\
=&\sum_{m=1}^n \sum_{l=0}^{n-m}\binom{n}{m} \frac{\lambda^{m-1}(1)_{m,1/\lambda}2^{n-m-1}(-1)^{l+n-1}}{m^{k-1}}S_{2,\lambda}(n-m,l)E_{l,\lambda}.
\end{split}
\end{equation*}
\end{theorem}

\begin{proof}
Replace $t$ by $1-e_{\lambda}(-2t)$, we observe
\begin{equation}\label{38}
\begin{split}
&\big(1-e_{\lambda}(-2t)\big)\sum_{m=0}^{\infty}E_{m,\lambda}^{(k)}\frac{(1-e_{\lambda}(-2t))^{m}}{m!}=\frac{l_{k,\lambda}(t)}{e_\lambda (1-e_{\lambda}(-2t))+ 1}\\
&=\frac{1}{2} \frac{2}{e_\lambda (1-e_{\lambda}(-2t))+ 1} \sum_{m=1}^\infty\frac {(1)_{m,1/\lambda}(-1)^{m-1} \lambda^{m-1}}{m^{k}(m-1)!}t^{m}\\
&=\frac{1}{2}\bigg(\sum_{l=0}^\infty E_{l,\lambda}(-1)^{l} \sum_{j=l}^\infty S_{2,\lambda}(j,l)\frac{(-2)^j t^j}{j!}\bigg)\sum_{m=1}^\infty\frac {(1)_{m,1/\lambda}(-1)^{m-1} \lambda^{m-1}}{m^{k-1}m!}t^{m}\\
&=\frac{1}{2}\sum_{j=0}^\infty \bigg(\sum_{l=0}^{j}E_{l,\lambda}(-1)^{l} S_{2,\lambda}(j,l)(-2)^j\bigg)\frac{t^j}{j!}\sum_{m=1}^\infty\frac {(1)_{m,1/\lambda}(-1)^{m-1} \lambda^{m-1}}{m^{k-1}}\frac{t^m}{m!} \\
&=\sum_{n=1}^\infty \bigg(\sum_{m=1}^n \sum_{l=0}^{n-m}\binom{n}{m} \frac{\lambda^{m-1}(1)_{m,1/\lambda}2^{n-m-1}(-1)^{l+n-1}}{m^{k-1}}S_{2,\lambda}(n-m,l)E_{l,\lambda}\bigg)\frac{t^n}{n!}.
\end{split}
\end{equation}

On the other hand, we have
\begin{equation}\label{39}
\begin{split}
&(1-e_{\lambda}(-2t))\sum_{m=0}^{\infty}E_{m,\lambda}^{(k)}\frac{(1-e_{\lambda}(-2t))^{m}}{m!}\\
&=-\sum_{i=1}^\infty\frac {(1)_{i,\lambda}(-2t)^{i}}{i!}\bigg(\sum_{m=0}^\infty E_{m,\lambda}^{(k)}(-1)^{m} \sum_{j=m}^\infty S_{2,\lambda}(j,m)\frac{(-2)^j t^j}{j!}\bigg)\\
&=-\sum_{i=1}^\infty\frac {(1)_{i,\lambda}(-2t)^{i}}{i!}\bigg(\sum_{j=0}^\infty\sum_{m=0}^{j} E_{m,\lambda}^{(k)}(-1)^{m}  S_{2,\lambda}(j,m)\frac{(-2)^j t^j}{j!}\bigg)\\
&=\sum_{n=1}^\infty\bigg(\sum_{i=1}^{n}\sum_{m=0}^{n-i}\binom{n}{i} (1)_{i,\lambda}2^{n}(-1)^{m+n+1}  E_{m,\lambda}^{(k)} S_{2,\lambda}(n-i,m)\bigg) \frac {t^{n}}{n!}.\\
\end{split}
\end{equation}

Therefore, by comparing the coefficients of  \eqref{38} and \eqref{39}, we have the desired result.

\end{proof}

\begin{theorem}For $n\geq 1, ~~k \in \mathbb{Z}$, we get
\begin{equation*}
\begin{split}
&\sum_{m=0}^n \sum_{l=0}^m \binom{n}{m} \lambda^{n-m-1} (1)_{n-m,1/\lambda}
 (-1)^{l+1} 2^{-l-1} S_{1,\lambda}(m,l)E_{l,\lambda}^{(k)}\\
&=\sum_{m=1}^n \sum_{l=0}^{n-m} \binom{n}{m} (-1)^{l-1}  2^{-l-1}  \frac{(1)_{m,1/\lambda} \lambda^{m-1} }{m^{k-1}}S_{1,\lambda}(n-m,l) E_{l,\lambda}.
\end{split}
\end{equation*}
\end{theorem}

\begin{proof}
Replace $t$ by $-\frac{1}{2}log_{\lambda}(1+t)$, from  \eqref{5},  \eqref{15},  we observe
\begin{equation}\label{40}
\begin{split}
&\frac{l_{k,\lambda}(-t)}{e_\lambda(-\frac{1}{2}log_{\lambda}(1+t))+1}\\
&=\frac{1}{2} \sum_{l=0}^\infty E_{l,\lambda} \frac{(-\frac{1}{2}log_{\lambda}(1+t))^l}{l!}  \sum_{m=1}^\infty \frac{(-\lambda)^{m-1}(1)_{m,1/\lambda}}{(m-1)!m^k}(-t)^m\\
&=\frac{1}{2} \sum_{l=0}^\infty E_{l,\lambda} (-\frac{1}{2})^l \sum_{j=l}^\infty S_{1,\lambda}(j,l) \frac{t^j}{j!} \sum_{m=1}^\infty \frac{(1)_{m,1/\lambda}(-\lambda)^{m-1}}{m^{k-1}}\frac{t^m}{m!}\\
&=\frac{1}{2}\sum_{j=0}^\infty \bigg( \sum_{l=0}^j (-1)^{l} 2^{-l}S_{1,\lambda}(j,l) E_{l,\lambda} \bigg) \frac{t^j}{j!} \sum_{m=1}^\infty \frac{(1)_{m,1/\lambda}(-\lambda)^{m-1}(-1)^{m}}{m^{k-1}}\frac{t^m}{m!}\\
&=\sum_{n=1}^\infty \bigg( \sum_{m=1}^n \sum_{l=0}^{n-m} \binom{n}{m} (-1)^{l-1}  2^{-l-1}  \frac{(1)_{m,1/\lambda} \lambda^{m-1} }{m^{k-1}}S_{1,\lambda}(n-m,l) E_{l,\lambda}\bigg)\frac{t^n}{n!}.
\end{split}
\end{equation}
On the other hand, from  \eqref{15}, we have
\begin{equation}\label{41}
\begin{split}
&-\frac{1}{2} log_{\lambda}(1+t) \sum_{l=0}^\infty  E_{l,\lambda}^{(k)}\frac{(-\frac{1}{2}log_\lambda(1+t))^l}{l!}\\
&=-\frac{1}{2}\sum_{j=1}^\infty \lambda^{j-1} (1)_{j,1/\lambda} \frac{t^j}{j!}  \sum_{l=0}^\infty  E_{l,\lambda}^{(k)} (-\frac{1}{2})^l \sum_{m=l}^\infty S_{1,\lambda}(m,l)\frac{t^m}{m!} \\
&=-\frac{1}{2}\sum_{j=1}^\infty \lambda^{j-1} (1)_{j,1/\lambda} \frac{t^j}{j!} \sum_{m=0}^\infty \bigg( \sum_{l=0}^m (-1)^{l} 2^{-l}S_{1,\lambda}(m,l)E_{l,\lambda}^{(k)}  \bigg)\frac{t^m}{m!} \\
&=\sum_{n=1}^\infty\bigg(\sum_{m=0}^n \sum_{l=0}^m \binom{n}{m} \lambda^{n-m-1} (1)_{n-m,1/\lambda} (-1) ^{l+1} 2^{-l-1} S_{1,\lambda}(m,l)E_{l,\lambda}^{(k)} \bigg)\frac{t^n}{n!}.
\end{split}
\end{equation}

Therefore, by comparing the coefficients  of  \eqref{40} and \eqref{41}, we have the desired result.

\end{proof}

\begin{theorem}For $n\geq 0, ~~k \in \mathbb{Z}$, we have
\begin{equation*}
\begin{split}
E_{n,\lambda}^{(k)}=\sum_{l=0}^n \sum_{m=0}^l \binom{n}{l} \binom{l}{m} \frac{(-1)^{l} 2^l (1)_{m+1,\lambda}}{m+1}\beta^{(k)}_{l-m,\lambda}E_{n-l,\lambda}
\end{split}
\end{equation*}
\end{theorem}

\begin{proof} From \eqref{3}, \eqref{4} and  \eqref{7},  we observe
\begin{equation}\label{42}
\begin{split}
\sum_{n=0}^\infty E_{n,\lambda}^{(k)}\frac{t^n}{n!}
&=\frac{l_{k,\lambda}(1-e_\lambda(-2t))}{t(e_\lambda(t)+1)}\frac{1-e_\lambda(-2t)}{1-e_\lambda(-2t)}\\
&= \frac{1}{2t} \frac{2}{e_\lambda(t)+1}\frac{l_{k,\lambda}(1-e_\lambda(-2t))}{1-e_\lambda(-2t)} (1-e_\lambda(-2t))\\
&=\frac{1}{2t}  \sum_{i=0}^\infty E_{i,\lambda} \frac{t^i}{i!}\sum_{j=0}^\infty \beta_{j,\lambda}^{(k)} \frac{(-2t)^j}{j!} \sum_{m=1}^\infty (-1)^{m+1} 2^m (1)_{m,\lambda} \frac{t^m}{m!} \\
&= \sum_{i=0}^\infty E_{i,\lambda} \frac{t^i}{i!}\sum_{j=0}^\infty \beta_{j,\lambda}^{(k)}(-2)^j \frac{t^j}{j!} \sum_{m=0}^\infty\frac{ (-1)^{m+2} 2^m (1)_{m+1,\lambda} }{m+1} \frac{t^m}{m!} \\
&= \sum_{i=0}^\infty E_{i,\lambda} \frac{t^i}{i!} \sum_{l=0}^\infty\bigg( \sum_{m=0}^l \binom{l}{m} \frac{(-1)^{l+2} 2^l (1)_{m+1,\lambda}}{m+1} \beta_{l-m,\lambda}^{(k)}\bigg)\frac{t^l}{l!}\\
&=\sum_{n=0}^\infty\bigg(\sum_{l=0}^n  \sum_{m=0}^l\binom{n}{l} \binom{l}{m} \frac{(-1)^{l+2} 2^l (1)_{m+1,\lambda}}{m+1} \beta_{l-m,\lambda}^{(k)}E_{n-l,\lambda} \bigg)\frac{t^n}{n!}.
\end{split}
\end{equation}

Therefore, by comparing the coefficients on both sides of  \eqref{42}, we have the desired result.
\end{proof}

\vspace{0.1in}

\noindent{\bf{Acknowledgments}} \\
%The authors would like to thank the referees for the detailed and valuable comments that helped improve the original manuscript in its present form.
The authors  thank Jangjeon Institute for Mathematical Science for the support of this research.\\
\vspace{0.1in}

\noindent{\bf{Funding}} \\
This research was funded by the National Natural Science Foundation of China (No. 11871317,
11926325, 11926321), Key Research and Development Program of Shaanxi (No. 2021GY-137).
%\vspace{0.1in}
%
%\noindent{\bf{Ethics approval and consent to participate}} \\
%The author  reveal that there is no ethical problem in the production of this paper.
%
\vspace{0.1in}

\noindent{\bf {Competing interests}} \\
The authors declare no conflict of interest.
%
%\vspace{0.1in}
%
%\noindent{\bf{Consent for publication}} \\
%All authors want to publish this paper in this journal.
%
\vspace{0.1in}

%\noindent{\bf{Author' Contributions}}\\

%\vspace{0.1in}

%\noindent{\bf{Author details}}\\


\vspace{0.1in}


\begin{thebibliography}{9}





%\reftitle{References}
%\begin{thebibliography}{999}
%\providecommand{\natexlab}[1]{#1}

\bibitem{1}
Carlitz, L.   {Degenerate Stirling, Bernoulli and Eulerian numbers,} \emph{Utilitas Math.} \textbf{1979}, {\emph{15}}, 51-88.

\bibitem{2}
Comtet, L. {Advanced Combinatorics: The Art of Finite and Infinite Expansions} (translated from the French by J.W.
Nienhuys), Reidel, Dordrecht and Boston, 1974.

\bibitem{3}
Dolgy, D.V.;  Kim, T. {Some explicit formulas of degenerate Stirling numbers associated with the degenerate special numbers and polynomials,} \emph{Proc. Jangjeon Math. Soc.} \textbf{2018}, {\emph{21(2)}}, 309-317.

\bibitem{4}  Hardy, G.H.
{On a class a functions,} \emph{ Proc. London Math. Soc.(2)} \textbf{1905}, \emph{3}, 441-460.

\bibitem{5} Jonqui$\grave{e}$re, A.
{Note sur la serie $\sum_{n=1}^\infty \frac{x^n}{n^s}$, } \emph{ Bull. Soc. Math. France} \textbf{1889}, {\emph{17}}, 142-152.

\bibitem{6} Jang, G.W.;   Kim, T. {A note on type 2 degenerate Euler and Bernoulli polynomials }, \emph{Adv. stud Contemp. Math}
\textbf{2019}, {\emph{29}},  147-159.

\bibitem{7} Kaneko, M.   {Poly-Bernoulli numbers},  \emph{J.Th$\acute{e}$or. Nombres Bordeaux}
\textbf{1997}, {\emph{9}}, 221-228.

\bibitem{8} Kim, D.S.; Kim, T.
{ A note on polyexponential and unipoly functions,} \emph{Russ. J. Math. Phys. }
\textbf{2019}, {\emph{ 26(1)}}, 40-49.

\bibitem{9} Kim, D.S.; Kim, T.
{ A note on a new  type of degenerate Bernoulli numbers,}
\emph{Russ. J. Math. Phys. }
\textbf{2020}, {\emph{ 27(2)}}, 227-235.

\bibitem{10}
Kim, D.S.;  Kim, H.Y.; Kim, D.; Kim, T.   {Identities of symmetry for type 2 Bernoulli and Euler polynomials,} \emph{Symmetry} \textbf{2020}, \emph{11}, 613.

\bibitem{11}
Kim, D.S.; Kim, T. {Some applications of degenerate poly-Bernoulli numbers and polynomials,} \emph{ Georgian Math. J.} \textbf{2019}, 415-421.

\bibitem{12}
Kim, H.K. A new type of degenerate poly-type 2 Euler polynomials and degenerate unipoly type 2 Euler polynomials, DOI:10.13140/RG.2.2.15812.86401

\bibitem{13}
Kim, T.  {A note on degenerate Stirling polynomials of the second kind,} \emph{ Proc. Jangjeon Math. Soc.}  \textbf{2017}, {\emph{20(3)}}, 319-331.

\bibitem{14}  Kim, T.;  Kim, D.S.;   Kwon, J.;  Lee,  H.S.
{Degenerate polyexponential functions and type 2 degenerate poly-Bernoulli numbers and polynomials, } \emph{Adv. Difference Equ.} \textbf{ 2020}, {\emph{ 168}}, 12pp.

\bibitem{15}   Kim, T.;  Kim, D.S.; Kim, H.Y.;  Jang, L.C.
{Degenerate poly-Bernoulli numbers and polynomials, } \emph{ Informatica }, \textbf{2020}, {\emph{31(3)}}, 1-7.

\bibitem{16}  Kim, T.;  Kim, D.S.;  Jang, L.C.;  Kim, H.Y.
{On type 2 degenerate Bernoulli and Euler polynomials of complex variable, } \emph{ Adv. Difference Equ. }, \textbf{2019}, {\emph{490}}, 15pp.

\bibitem{17} Kim, T.;  Jang, L.C.; Kim, D.S.;  Kim, H.Y.
{ Some identities on type 2 degenerate Bernoulli polynomials of the second kind},  \emph{ symmetry},
\textbf{2020}, {\emph{12}},  510.

\bibitem{18} Kim, T.;  Kim, D.S.;  Jang, L.C.;  Lee, H.
 Jindalrae and Gaenari numbers and polynomials in connection with Jindalrae-Stirling numbers, \emph{ Adv. Difference Equ.} \textbf{2020},   \emph{245(2020)}.

\bibitem{19} Kim, T.;  Kim, D.S.
 Some relations of two type 2 polynomials and discrete harmonic numbers and polynomials, \emph{Symmetry}, \textbf{2020}, {\emph{12(6)}}, 905.

\bibitem{20} Kim, T.;  Kim, D.S. { Degenerate Laplace transform and degenerate gamma function}, \emph{Russ. J. Math. Phys.} \textbf{2017}, {\emph{24(2)}},   241-248.

\bibitem{21} Kim, T.;  Kim, D.S. { Note on the degenerate gamma function}, \emph{Russ. J. Math. Phys.} \textbf{2020}, {\emph{27(3)}},   352-358.

\bibitem{22} Kim, T. { $\lambda$-Analogue of Stiring polnomials of the second kind}, \emph{Proc. Jangjeong Math. Soc.} \textbf{2017}, {\emph{20(3)}},   319-331.

\bibitem{23}
Lee, D.S.; Kim, H.K.; Jang, L.C. Type 2 Degenerate Poly-Euler Polynomials, \emph{Symmetry}, \textbf{2020}, {\emph{12}}, 1011.

\bibitem{24} Lewin, L.
{Polylogarithms and associated functions, } With a foreword by A. J. Van der Poorten. \emph{North-Holland Publishing Co. New York-Amsterdam,} \textbf{1981}, xvii+359pp.

\bibitem{25}
Ma, Y.; Kim, D. S.; Lee, H.; Kim, T. Poly-Dedekind sums associated with poly-Bernoulli functions,\emph{ J. Inequal.
 Appl.}\textbf{ 2020},  {\emph{(2020)}}: 248, 10 pp.

\bibitem{26} Ohno, Y.;  Sasaki, Y. {On the party of poly-Euler numbers},  \emph{RIMS K$\hat{o}$ku$\hat{u}$roku Bessatsu},
\textbf{2012}, {\emph{B32}}, 271-278.

\bibitem{27} Yoshinori, H.   {Poly-Euler polynomials and Arakawa-Kaneko type zeta funtions},  \emph{Funtiones et Approximatio},
\textbf{2014}, {\emph{5.1.1}}, 7-22




\end{thebibliography}
\end{document}